\documentclass[12pt]{amsart}
\usepackage{amsmath}

\def\e{\epsilon}
\def\lf{\left}
\def\ri{\right}

\def\p{\partial}

\def\jbar{{\bar\jmath}}

\def\K{K\"ahler }

\def\KRF{K\"ahler-Ricci flow }

\def\A{Amp\`{e}re }
\def\H{H\"{o}lder }
\def\be{\begin{equation}}
\def\ee{\end{equation}}

\def\lf{\left}
\def\ri{\right}

\def\e{\epsilon}

\def\ijb{{i\jbar}}

\def\p{\partial}

\def\p{\partial}

\def\p{\partial}

\def\KRF{K\"ahler-Ricci flow }

\newtheorem{thm}{Theorem}

\newtheorem{lem}{Lemma}

\newtheorem{cor}{Corollary}
\theoremstyle{definition}

\theoremstyle{remark}

\begin{document}

\title{A $C^0$-estimate for the parabolic Monge-Amp\`{e}re equation on complete non-compact K\"ahler manifolds}
\author{Albert Chau$^1$}

\address{Department of Mathematics,
The University of British Columbia, Room 121, 1984 Mathematics
Road, Vancouver, B.C., Canada V6T 1Z2} \email{chau@math.ubc.ca}

\author{Luen-Fai Tam$^2$}

\thanks{$^1$Research
partially supported by NSERC grant no. \#327637-06}
\thanks{$^2$Research partially supported by Hong Kong RGC General Research Fund  \#GRF 2160357}

\address{The Institute of Mathematical Sciences and Department of
 Mathematics, The Chinese University of Hong Kong,
Shatin, Hong Kong, China.} \email{lftam@math.cuhk.edu.hk}

\date{July 2008}

\maketitle\markboth{Albert Chau and Luen-Fai Tam} {$C^0$-estimate
for the parabolic complex Monge-Amp\`{e}re equation }
\begin{abstract}

In this article we study  the \K Ricci flow, the corresponding parabolic Monge \A
equation  and complete non-compact \K Ricci flat
manifolds.  In our main result Theorem \ref{mainthm} we prove that
if $(M, g)$ is sufficiently close to being \K Ricci flat in a
suitable sense, then the \KRF \eqref{KRF} has a long time smooth
solution $g(t)$ converging smoothly uniformly on compact sets to a
complete \K Ricci flat metric on $M$.  The main step is to obtain
a uniform $C^0$-estimates for the corresponding parabolic Monge \A
equation. Our results on this can be viewed as a parabolic version
of the main results in \cite{TY3} on the elliptic Monge \A
equation.

\end{abstract}

{\bf 1.} Let $(M^n,g_0)$ be a complete non-compact \K manifold
with complex dimension $n$ and consider the following \KRF on $M$:

 \begin{equation}\label{KRF}
    \left\{%
\begin{array}{ll}
    \dfrac{\p   {g}_{i\jbar}}{\p t} =- R_{i\jbar}
      \\
     g_\ijb(x, 0) =(g_0)_\ijb \\
\end{array}%
\right.
\end{equation}

We are interested in studying when \eqref{KRF} admits a long time
solution  $g(t)$ converging smoothly on $M$ to a complete
\K metric $g(\infty)$.  Such a limit $g(\infty)$ must be \K
Einstein with zero scalar curvature by \eqref{KRF}.  We are thus
interested in studying when  $(M^n,g_0)$ converges to a \K Ricci
flat metric under \eqref{KRF}.   When $M$ is compact Cao \cite{Ca}
established that a necessary and sufficient condition for such
convergence is that:

\begin{equation}\label{Riccipotential} (R_0)_{i\jbar} =(f_0)_{i\jbar}\end{equation}
for a smooth potential function $f_0$ on $M$, thus re-establishing
the  famous Calabi Conjecture first proved by Yau
\cite{Yau-1978-Ricci-flat}.  In Theorem \ref{mainthm} we establish
a non-compact version of Cao's result. We prove that when
$(M^n,g_0)$ is complete, non-compact with bounded curvature,
  with volume growth $V_{x_0}(r)\le Cr^{2n}$ for some $x_0$  and $C$ for all $r$,
   and satisfyies a Sobolev inequality, then:

 \begin{it} Under the above conditions, the \KRF  \eqref{KRF} has a long
 time solution $g(t)$ converging smoothly on $M$ provided
 \eqref{Riccipotential}
 is satisfied and
 $|f_0|(x)\le \frac{C}{1+\rho_0^{2+\e}(x)}$ for
 some $C, \e>0$ and all $x$.\end{it}

 See Theorem \ref{mainthm} for details.   The result  is motivated by 
 the work of Tian-Yau  \cite{TY3}  who  proved the existence of a \K Ricci flat
  metric in the compliment of a divisor in a projective variety
  under certain conditions. They first constructed a K\"ahler
  metric satisfying the condition in Theorem \ref{mainthm} and
  then solve the  elliptic complex Monge \A equation to obtain a
  \K Ricci flat metric. Our results
    here can be viewed as a parabolic version of the results on the elliptic
     Monge \A equation in \cite{TY3}.
      We also refer to \cite{Chau,Ca}  for results on convergence of the \KRF to K\"ahler Eienstein metrics with negative scalar curvature. In \cite{Ca}, it was
proved that (1) converges after rescaling
 to a K\"ahler Eienstein metric with negative scalar curvature provided
 $(R_0)_{i\jbar}+(g_0)_{i\jbar}=(f_0)_{i\jbar} $ for smooth $f_0$.  A
 non-compact version of this result was proved in \cite{Chau}.
\vskip .5cm


{\bf 2.} Let $(M^n,g_0)$ be a complete non-compact \K manifold
with complex dimension $n$ such that \eqref{Riccipotential} holds for
some smooth potential $f_0$ on $M$.  We want to study the long
time behavior of the \KRF \eqref{KRF}, when $f_0$ is sufficiently
close to zero in a suitable sense.  Note that $(M^n, g_0)$ is \K
Einstein with zero scalar curvature when $f_0=0$.  We are thus
interested in the behavior of the \KRF on complete \K manifolds
which are close to being \K Einstein.   We
will prove the following:

\begin{thm}\label{mainthm} Let $(M^n,g_0)$ be a complete non-compact
\K manifold with bounded curvature and $n\ge3$.   Assume the following:
\begin{enumerate}
       \item [(a)] Condition \eqref{Riccipotential} holds for smooth $f_0$  satisfying:

        \begin{equation}\label{decaycondition}
        |f_0|(x)\le \frac{C_1}{1+\rho_0^{2+\e}(x)}
\end{equation}
   for some $C_1 , \e >0$,  and all $x\in M$ where
    $\rho_0(x)$ is the distance function from a fixed $p\in M$    \item [(b)] The following Sobolev inequality is true:
    \begin{equation}\label{Sobolev}
\lf(\int_M|\phi|^{\frac{2n}{n-1}}dV_0\ri)^{\frac{n-1}{n}}\le
C_2\int_M|\nabla_0 \phi|^2
\end{equation}
 for some $C_2>0$ and all $\phi\in C_0^\infty(M)$.
\item[(c)] There exists constant $C_3>0$ such that
\begin{equation}\label{volumegrowth}
    V_0(r)\le C_3r^{2n}
\end{equation}
for some $C_3 >0$ and all $r$ where $V_0(r)$ is the volume of the geodesic
ball with radius $r$  centered at some $p\in M$.
\end{enumerate}
Then \eqref{KRF} has long time solution $g(t)$ converging
uniformly on compact sets in the $C^\infty$ topology  to a complete
 \K Ricci flat metric $g_{\infty}$ on $M$.
\end{thm}

 We will also consider the following
parabolic Monge Amp\`ere equation corresponding to (\ref{KRF}):

\begin{equation}\label{MA}
\left\{%
\begin{array}{ll}
    \dfrac{\p u}{\p t}= \log \dfrac{\det ((g_0)_{k \bar{l}}+ u_{k
\bar{l}})}{\det ((g_0)_{k \bar{l}})}-f_0
      \\
    u(x,0)=0 \\
\end{array}%
\right.
\end{equation}
The relationship between the two equations can be described simply as follows.  If \eqref{MA}  has a smooth solution $u$ on $M\times [0,T)$ then $g_\ijb=(g_0)_\ijb+u_\ijb$ is a smooth solution to \eqref{KRF}. Conversely, if \eqref{KRF} has a solution $ g$ then

$$
u(x,t)=\int_0^t\log\frac{\det( g_\ijb)}{\det( (g_0)_\ijb)}(x,s)ds-tf_0(x)
$$
is a solution to \eqref{MA} as shown in \cite{Chau}.

\begin{lem}\label{ut-l1} Let $(M^n,g_0)$ be a complete non-compact
\K manifold with bounded curvature such that \eqref{Riccipotential} holds for a smooth bounded potential $f_0$. Suppose $g(t)$
is a smooth solution to \eqref{KRF} on $M\times[0,T)$ such that $g(t)$ has bounded curvature for each $0<t<T$.
 Then if $u(x,t)$ is the corresponding solution to \eqref{MA}, then $f(t)=-u_t$ satisfies the following for each $0<t<T$:
\begin{enumerate}
    \item [(i)] $ |f|$, $t|\nabla_t f|^2$ and $t\Delta_t f$ are all
        uniformly bounded on $M$ by a constant
        depending only on $f_0$ and  $g_0$.
    \item [(ii)] All the covariant derivatives of
        $f(t)$ relative to $g(t)$ are bounded on $M$.
\item[(iii)] If in addition,  $|f_0(x)|\le C/(1+\rho_0(x))^N$ for some
    $N>1$ and distance function $\rho_0$ relative some $p\in M$, then there is a constant $C_t$ such
    that $|f(x,t)|\le C_t/(1+\rho_t(x))^N$ where $\rho_t(x)$ is the
    distance relative to $p$ with respect to $g(t)$.
\end{enumerate}
\begin{proof} Fix some $t_0\in (0, T)$.  We begin by noting that by (\ref{KRF}) and the uniform curvature bounds, $g(t)$ is uniformly equivalent to $g_0$ on $M\times [0,t_0]$.  Moreover, by \cite{Sh2} all covariant derivatives of $g(t)$ are bounded on $M\times (0, t_0]$.  In particular, it follows that $(M, g(t))$ has bounded geometry of all orders for each $t\in (0, t_0]$.

Now differentiating (\ref{MA}) gives

\begin{equation}\label{ut-e1}
   \lf(\frac{\p}{\p t}-\Delta_t \ri) f=0.
\end{equation}
As $f_0$ is bounded, it follows from the above bounds on $g(t)$ (and curvature) and the maximum principle in \cite[Lemma 4.7]{Sh2} that $\sup_M |f(t)| \le \sup_M |f_0|$ for all $t\in [0, t_0]$.  From this, the fact that  $\Delta_t f(t)= R(t)$, the bounded geometry of $g(t)$ mentioned above and standard elliptic theory, we see that (ii) is true since $t_0\in (0, T)$ was arbitrarily chosen.  Part (i) now follows from the uniform bound on $|f(t)|$, the bounded geometry of $g(t)$ mentioned above, the main Theorem in \cite{Ch} and by letting $t_0 \to T$.

We now prove (iii).  By Lemma  4.5 in \cite{Sh2} there exists  a
smooth exhaustion function $\rho_0$ such that for some positive
constants $K_1, K_2, K_3$, depending only on $n$, $t_0$ and the uniform bound on $Rm(x, t)$ on $M\times [0, t_0]$, satisfying
\begin{enumerate}
\item $K_1 (r_t+1) \leq \rho_0 \leq K_2 (r_t+1)$ \item
    $|\nabla_t \rho_0|, |\nabla^2_t \rho_0| \leq K_3$.
\end{enumerate}
for all $t\in [0, t_0]$.  Direct computation gives
\begin{equation}\label{s1e3}
\begin{split}
\frac{d}{dt} (\rho_0 ^{2N} f^2)&\le  \rho_0 ^{2N} \Delta_t f^2\\
&=\Delta_t(\rho_0 ^{2N} f^2)-(\Delta_t \rho_0 ^{2N})f^2-2
\langle\nabla_t\rho_0 ^{2N}, \nabla_t f^2\rangle_t
 \\
&=\Delta_t( \rho_0 ^{2N}f^2)-2N \rho_0^{2N-1}
(\Delta_t\rho_0) f^2 \\
&\quad-2N(2N-1)\rho_0^{2N-2} |\nabla_t
\rho_0|_t^2 f^2
 -4N\rho_0 ^{2N-1} \langle\nabla_t \rho_0, \nabla_t
f^2\rangle_t
\\&\le \Delta_t(\rho_0^{2N} f^2)+C_1\rho_0 ^{2N} f^2\\
&\quad-4N  \rho_0 ^{2N-1}
\langle\nabla_t \rho_0, (\frac{\nabla_t(\rho_0^{2N} f^2)-\nabla_t
(\rho_0^{2N})f^2}{\rho_0^{2N}})\rangle_t\\
&\le \Delta_t(\rho_0^{2N} f^2)+C_2\rho_0^{2N} f^2-4N \rho_0
^{-1}
\langle \nabla_t \rho_0,  \nabla_t (\rho_0^{2N}f^2 )\rangle_t\\
\end{split}
\end{equation}
for constants $C_1 , C_2>0$ depending only on $n, N, t_0$
and the  bound on $Rm(x, t)$ on $M\times [0, t_0]$.

On the other hand, for any $t\in [0,t_0]$ and for any $R>0$, let
$\varphi$ be a cut off function on $B_0(R)$ which is zero outside
$B_0(2R)$, and $|\nabla_0\varphi|_0\le C/R$ for some constant $C$
independent of $t_0$ and $R$. Then
\begin{equation}\label{s1e3-1}
   \begin{split}
   \int_{M}\frac{\p}{\p t}(\varphi^2f^2)dV_t&=
   \int_{M}\varphi^2\lf(\Delta_t f^2 dV_t -2|\nabla_t f|_t^2\ri)dV_t\\
   &\le2\int_{M}|\nabla_t\varphi|_t^2f^2dV_t+2\int_{M} \phi \Delta_t\phi f^2dV_t\\
   &\quad -\int_{M}\varphi^2|\nabla_t f|_t^2dV_t\\
   \end{split}
\end{equation}
Integrating from $0$ to $t_0$ and using the uniform bound on $Sup_M |f(t)|$, we have
\begin{equation}\label{s1e3-2}
  \int_0^{t_0} \int_{B_0(R)}\varphi^2|\nabla_t f|_t^2dV_t dt\le
  C_3V_t(B_0(R))\le C_4 \exp(aR^2)
\end{equation}
for some constants $C_3$, $C_4$ depending only on the  bound on $Rm(x, t)$ on $M\times [0, t_0]$, $n$ and $t_0$.  In
particular, we can find $b>0$ such that
\begin{equation}\label{s1e3-3}
    \int_0^{t_0} \int_M \exp(-b\rho_t^2)|\nabla_t (\rho_0^{2N}f^2)|_t^2dV_t
    dt<\infty.
\end{equation}
Let $h=\exp(-C_2t)\rho_0^{2N}f^2$, then by (\ref{s1e3})
\begin{equation}\label{s1e3-4}
\frac{\p}{\p t}h\le \Delta_t h-4N \rho_0 ^{-1} \langle \nabla_t
\rho_0,  \nabla_t h\rangle_t.
\end{equation}
Since at $t=0$, $h$ is bounded, if we let $h_1=h-\sup_{M}h(0)$,
then $h_1$ still satisfies (\ref{s1e3-4}) and $h_1(0)\le 0$. By
Theorem 4.3 in \cite{EH}, together with \eqref{s1e3-3},
\eqref{s1e3-4}, and the properties of $\rho_0$, we conclude that
$h(t)\le 0$ for all $t\in [0, t_0]$.  Part (iii) of the lemma now follows by letting $t_0 \to T$.   Note that we have used
the fact that $g $ and $g(t)$ are uniformly equivalent.

This completes the proof of the lemma.
\end{proof}
\end{lem}

\begin{cor}\label{longtime-c1} Let $(M^n, g_0)$ be a complete
non-compact \K manifold with bounded curvature such that \eqref{Riccipotential} is satisfied for a smooth bounded potential $f_0$. Then \eqref{KRF}
has a long time smooth solution.
\end{cor}
\begin{proof} By \cite{Sh2}, \eqref{KRF} has a solution on $M\times[0,T)$ where $T>0$ is the maximal time satisfying the assumptions in  Lemma \ref{ut-l1}. Let $u$ and $f$
 as in the lemma. Then $|f(t)|\leq sup_M|f_0|$ and hence $|u(x,t)|\le t
 \sup_M|f_0|$ for all $(x, t) \in M\times[0, T)$ by  Lemma \ref{ut-l1}.  Hence $\sup_M u(x, t)$ cannot blowup in finite
 time.  By the a prior estimates in \cite{Chau}, we
 conclude that $T=\infty$ and thus \eqref{KRF} has a long time smooth solution.
 \end{proof}
Note the solution $g(t)$ in the corollary satisfies the assumptions
in Lemma \ref{ut-l1} for any finite $t>0$, and hence we may assume in Theorem \ref{mainthm} that
$f_0$ has uniformly bounded covariant derivatives of all orders.

With the assumptions as in Theorem \ref{mainthm}, let $g(t)$ be
the long time solution as in Corollary \ref{longtime-c1} and let
$u(x, t)$ and $f(x, t)$ be defined as before. In order to prove Theorem \ref{mainthm}, we
need the following lemmas.

\begin{lem}\label{uestimates-l1} Let $(M,g_0)$ be a complete non-compact \K manifold with bounded curvature satisfying conditions (a) and (c) in
Theorem \ref{mainthm} with $n\ge 3$ and let $g(t)$ be the longtime
solution of the \K Ricci flow \eqref{KRF}. For any $p_0>n$
with $(p_0+1)(2+\e)/(n+p_0)>2$ and $p^*=n(p_0+1)/(n+p_0)>2$, there
is a positive constant $C$ independent of $t$ such that
$$
\int_M|f(t)|^{p^*}dV_t\le C
$$
for all $t$.
\end{lem}
\begin{lem}\label{uestimates-l2} Let $(M,g_0)$ be as in
Theorem \ref{mainthm}  and let $g(t)$ be the longtime solution of
the K\"ahler-Ricci flow \eqref{KRF}. Then there is a  positive constant
$C$   such that
 and
$$
\sup_{M\times[0,\infty)}|u|\le C.
$$
for all $t$.
\end{lem}
Assuming the lemmas are true. Then we can prove Theorem
\ref{mainthm} as follows:

\begin{proof}[Proof of Theorem \ref{mainthm}] By Lemmas \ref{uestimates-l1} and \ref{uestimates-l2} and the a priori estimates in \cite{Chau}, we conclude that for all $t\geq0$, $g(t)$ is uniformly
equivalent to $g_0$ independent of $t$.  Moreover, for any sequence
$t_k\to\infty$ some subsequence of $u(x,t_k)$ (which we still denote by  $u(x,t_k)$)  converges in the  $C^\infty$ sense on
compact subsets of $M$ to a smooth limit $v$ on $M$. Thus by Lemma
\ref{ut-l1} (i),  we
conclude that  $\frac{\p u}{\p t}(t_k)=-f(t_k)$ converges uniformly on $M$ to
a constant $c$.  Lemma \ref{uestimates-l1} and the fact
that $M$ has infinite volume implies that $c$ must be zero.  Hence $g_\ijb+v_\ijb$ is a smooth complete Ricci flat K\"ahler
metric on $M$.  The proof of the theorem
is complete once we show that the limit metric $g_\ijb+v_\ijb$ is independent of
the $t_k$.   This last fact follows by the next lemma.
\end{proof}

The following result was basically proved in \cite{Chau05}.

\begin{lem}\label{unique-lem} Let  $g$ and $h$ be two equivalent complete  K\"ahler metrics on a non-compact complex manifold $M$
 such that: (i)
 $h_\ijb=g_\ijb+v_\ijb$ for some smooth bounded function $v$; (ii) $g$ has nonnegative
  Ricci curvature; and (iii)
  $$
  \log \frac{\det (g_\ijb+v_\ijb )}{\det (g_\ijb)}=0.
  $$
  Then $g=h$.
\end{lem}
\begin{proof} We sketch the proof.  Since
\begin{equation*}
\begin{split}
  0 & =\int_0^1\frac{\p }{\p s}\log \det (g_\ijb+sv_\ijb
)  \\
    &=\lf( \int_0^1 g^{i\bar j}(s)ds\ri)v_\ijb
\end{split}
\end{equation*}
where $g_\ijb(s)= g_\ijb+sv_\ijb$. Hence $v$ satisfies $a^{i\bar
j}v_\ijb=0$ for some K\"ahler metric $a_\ijb$ which is uniformly
equivalent to $g$. As $g$ has nonnegative Ricci curvature, $v$
must be constant by \cite{gr,Saloff-Coste92}.
\end{proof}

We now prove Lemma \ref{uestimates-l1}.

\begin{proof} [Lemma \ref{uestimates-l1}] Let $p_0>n$ be such that
$$
\frac{p_0+1}{n+p_0}(2+\e)>2
$$
and  $p^*=\frac{n(p_0+1)}{n+p_0}>2$.

By the assumption on $f_0$, the volume growth of $(M,g_0)$ and
Lemma \ref{ut-l1}, we see that for any    $T>0$ and $p\ge p^*$,
there is a constant $C$
\begin{equation}\label{mainlem-e1}
\int_M |f(t)|^{p}dV_t\le C
\end{equation}
for all $0\le t\le T$. Here we have used the fact that for any
finite $T$, $ g(t)$ is uniformly equivalent to $g$ for $0\le t\le
T$.

For any even integer $p>p^*+2$,  we have
\begin{equation}\label{s1e4}
\begin{split}
\frac{d}{dt}\int_M |f|^p dV_t=&\frac{d}{dt}\int_M  f^p dV_t\\
&p\int_M  f^{p-1} \Delta_t f dV_t-\int_M  f^p\Delta_t f dV_t\\
=&-p(p-1)\int_M  f ^{p-2}|\nabla_t f|^2_t dV_t +p\int_M  f^{p-1}|\nabla_t f|^2_t   dV_t\\
\le&-\lf(p(p-1) -  p C_1\ri)\int_M  f^{p-2}|\nabla_t f|^2_t
dV_t.
\end{split}\end{equation}
where $C_1=\sup_{M}|f_0|$ and we have used the fact that for any
$t$, $|\nabla f|_t$ is bounded on $M$ and the fact that $p-2>p^*$
to justify the integration by parts. Hence if $p\geq 1+C_1$, then
\begin{equation}\label{mainlem-e2}
\int_M |f(t)|^p dV_t\le \int_M |f_0|^pdV_0<\infty
\end{equation}
We have to improve \eqref{mainlem-e2}.

Let $v=\max\{f,0\}$ and let $\varphi$ be a nonnegative cutoff
function on $M$ independent of $t$. For any $p\ge p^*-1>1$,
\begin{equation}\label{ppart1}
   \begin{split}
   \int_0^T\int_M&\varphi^2 v^p\frac{\p}{\p t}f dV_t
   dt\\
   &=\int_0^T\int_M\varphi^2 v^p\Delta_t fdV_tdt\\
   &=-\int_0^T\int_M pv^{p-1} \varphi^2|\nabla_t v|_t^2 dV_tdt
   -2\int_0^T\int_M v^p \varphi\langle \nabla_t\varphi, \nabla_t v\rangle_t
   dV_tdt\\
   &\le \frac1p\int_0^T\int_M v^{p+1}  |\nabla_t\varphi|_t^2dV_tdt.
   \end{split}
\end{equation}
On the other hand,
\begin{equation}\label{ppart2}
\begin{split}
\int_0^T\int_M&\varphi^2 v^p\frac{\p f}{\p t}  dV_t
   dt\\
   &=-\int_M\varphi^2\int_0^Tv^p(\frac{\p}{\p v
   t})e^{ f_0-f}dtdV_0\\
        &=\frac1{p+1}\int_M\varphi^2 \int_0^T
        \frac{\p}{\p t}\lf[v^{p+1} e^{f_0-f}\ri]dtdV_0\\
        &\quad+\frac1{p+1}\int_M\varphi^2 \int_0^T
v^{p+1} e^{f_0-f}\frac{\p f}{\p t} dtdV_0.
\end{split}
\end{equation}
Combining this with (\ref{ppart1}), we have
\begin{equation}\label{ppart2}
\begin{split}
 \int_M\varphi^2 v^{p+1}e^{f_0-f}dV_t|_{t=T}
&\le \int_M\varphi^2
v^{p+1}dV_t|_{t=0}+\frac{p+1}p\int_0^T\int_M
v^{p+1} | \nabla_t\varphi|_t^2 dV_tdt\\
   &\quad-\int_M\varphi^2 e^{ f_0} \int_0^T
v^{p+1} \sum_{k=0}^\infty  \frac{(-f)^k}{k!}\frac{\p v}{\p
t}
dtdV_0\\
&=\int_M\varphi
v^{p+1}dV_t|_{t=0}+\frac{p+1}p\int_0^T\int_M
v^{p+1} | \nabla_t\varphi|_t^2dV_tdt\\
   &-\int_M\varphi^2 e^{f_0}
  \sum_{k=0}^\infty  \frac{(-1)^kv^{p+k+2}}{k!(p+k+2)}
dV_0|_{t=T}\\
&\quad +\int_M\varphi^2 e^{f_0}
  \sum_{k=0}^\infty  \frac{(-1)^kv^{p+k+2}}{k!(p+k+2)}dV_0|_{t=0}.
\end{split}
\end{equation}
Now

 \begin{equation}\label{ppartc3}
\begin{split}
    &-\int_M\varphi^2 e^{f_0}
  \sum_{k=0}^\infty  \frac{(-1)^kv^{p+k+2}}{k!(p+k+2)}
dV_0|_{t=T} +\int_M\varphi^2 e^{f_0}
  \sum_{k=0}^\infty  \frac{(-1)^kv^{p+k+2}}{k!(p+k+2)}dV_0|_{t=0}\\
   &= -\int_M\varphi^2 v^{p+2}e^{f_0}
  \sum_{k=0}^\infty  \frac{(-1)^kv^{k}}{k!(p+k+2)}
dV_0|_{t=T} \\
&\quad +\int_M\varphi^2 v^{p+2}e^{f_0}
  \sum_{k=0}^\infty  \frac{(-1)^kv^{k}}{k!(p+k+2)}dV_0|_{t=0}\\
&\le C_2\int_M  v^{p+2} dV_t|_{t=T}+C_3\int_M  v^{p+2} dV_t|_{t=0}\\
&\le C_2\int_M  v^{p+2} dV_t|_{t=T}+C_4(p)
\end{split}
\end{equation}
where $C_2, C_3, C_4$ are independent of $T$. Here we have used
the fact that $f(t)$ are uniformly bounded in spacetime. Combine
\eqref{ppart2}  and (\ref{ppartc3}) to give

 \begin{equation}\label{ppartc4}
 \begin{split}
\int_M\varphi^2 v^{p+1}dV_t|_{t=T}&\le
C_5+\frac{p+1}p\int_0^T\int_M v^{p+1}|
\nabla_t\varphi|_t^2dV_tdt\\
&\quad+C_2\int_M  v^{p+2} dV_t|_{t=T}\\
\end{split}
\end{equation}
for some constant independent of $T$. Letting $\varphi$
approach the constant function 1 gives

\begin{equation}\label{ppart3}
\int_M  v^{p+1}dV_t|_{t=T}\le C_5+C_2\int_M  v^{p+2}dV_t|_{t=T}.
\end{equation}

Similarly    one can prove that if $w=\max\{-f,0\}$, then
\begin{equation}\label{ppart4}
 \int_M  w^{p+1}dV_t|_{t=T}\le C_5+C_2\int_M  w^{p+2}dV_t|_{t=T}
\end{equation}
by modifying $C_5$ and $C_2$ if necessary, while still independent of $T$.
Hence we have
\begin{equation}\label{ppart5}
 \int_M  |f|^{p+1}dV_t|_{t=T}\le 2C_5+C_2\int_M  |f|^{p+2}|_{t=T}
\end{equation}
for all $p\ge p^*-1$. By iteration and \eqref{mainlem-e2}, we
conclude that
\begin{equation}\label{ppart5}
 \int_M  |f|^{p^*}dV_t|_{t=T}\le C_5
\end{equation}
for some constant $C_5$ independent of $T$.
\end{proof}

\begin{proof}[Proof of Lemma \ref{uestimates-l2}]
 We may now proceed as in \cite{TY3} using \eqref{ppart5}.  In particular note that for all $t$, $|u|$ also decays like
$r_t^{-2-\e}$ by Lemma \ref{ut-l1} and hence we may integrate by parts as in \cite{Ca}, using \eqref{ppart5} and the Sobolev inequality (\ref{Sobolev}),  to obtain the following for
 $p\ge p_0>p^*$: (here $p_0$ and $p*$ are as in Lemma 2)
\begin{equation}\label{iteration1}
\begin{split}
   \lf(\int_M |u|^{(p+1)\kappa}dV_0\ri)^\frac1\kappa &\le
   C_6p\int_M|u|^p|e^{f_0-f}-1|dV_0\\
   &\le C_7p\int_M|u|^p|f_0-f|dV_0\\
   &\le C_7p\lf(\int_M|u|^{p+1}dV_0\ri)^{\frac{p}{p+1}}
   \lf(\int_M|f_0-f|^{p+1}dV_0\ri)^{\frac1{p+1}}\\
   &\le C_8p\lf(\int_M|u|^{p+1}dV_0\ri)^{\frac{p}{p+1}}
\\&\le C_8p\lf(\int_M|u|^{p+1}dV_0+1\ri).
   \end{split}
\end{equation}
Here $C_8$ is a constant independent of $t, p$ and
$\kappa=n/(n-1)>1$ and we have used \eqref{ppart5} and the fact
that $f(t)$ is uniformly bounded on space and time.
Take $p=p_0$,
we also have:
\begin{equation}\label{iteration2}
\begin{split}
&\lf(\int_M |u|^{(p_0+1)\kappa}dV_0\ri)^\frac1\kappa\\
&\le C_7p_0\int_M|u|^{p_0}|f_0-f|dV_0\\
   &\le
C_7p_0\lf(\int_M|u|^{(p_0+1)\kappa}dV_0\ri)^{\frac{p_0}{(p_0+1)\kappa}}
\lf(\int_M|f_0-f|^{\frac{(p_0+1)\kappa}{(p_0+1)\kappa-p_0}}dV_0\ri)^{1-\frac
{p_0}{(p_0+1)\kappa}}\\
&=C_7p_0\lf(\int_M|u|^{(p_0+1)\kappa}dV_0\ri)^{\frac{p_0}{(p_0+1)\kappa}}
\lf(\int_M|f_0-f|^{p^*} dV_0\ri)^{\frac1 {p^*}}.
\end{split}
\end{equation}
Hence by Lemma 2 we have
\begin{equation}\label{iteration3}
  \lf(\int_M |u|^{(p_0+1)\kappa}dV_0\ri)^\frac1{(p_0+1)\kappa} \le
  C_9
\end{equation}
for some constant $C_9$ independent of $t$. By (\ref{iteration1})
and Young's inequality we have that for $p\ge p_0$:
\begin{equation}\label{iteration4}
\begin{split}
   \int_M |u|^{(p+1)\kappa}dV_0+1&\le (C_8p)^\kappa
   \lf[ \int_M|u|^{p+1}dV_0+1\ri]^\kappa+1\\
   &\le (C_8p)^\kappa\lf[\int_M|u|^{p+1}dV_0+2\ri]^\kappa\\
   &\le (2C_8p)^\kappa\lf[\int_M|u|^{p+1}dV_0+1\ri]^\kappa.
\end{split}
\end{equation}
Hence we have
\begin{equation}\label{iteration5}
\begin{split}
\lf[\int_M |u|^{(p+1)\kappa}dV_0+1\ri]^{\frac1{\kappa(p+1)}}&\le
(2C_8p)^{\frac1{p+1}}
\lf[\int_M|u|^{p+1}dV_0+1\ri]^{\frac1{p+1}}\\
&\le (2C_8(p+1) )^{\frac1{p+1}}
\lf[\int_M|u|^{p+1}dV_0+1\ri]^{\frac1{p+1}},
\end{split}
\end{equation}
 That is to say, for all $q\ge p_0+1$,

\begin{equation}\label{iteration6}
\lf[\int_M |u|^{ \kappa q}dV_0+1\ri]^{\frac1{\kappa q}} \le
  (2C_{8} q )^{\frac1{q}}
\lf[\int_M|u|^{q}dV_0+1\ri]^{\frac1{q}}.
\end{equation}
 By iteration (see \cite{TY3}), it is straight forward to show that
$$
\sup_M|u|\le C_{10}
\lf[\int_M|u|^{(p_0+1)\kappa}dV_0+1\ri]^{\frac1{(p_0+1)\kappa}}\le
C_{11}
$$
for some constants $C_{10}, C_{11}>0$ independent of $t$. Here we have used
(\ref{iteration3}). This completes the proof of the lemma.
\end{proof}

\bibliographystyle{amsplain}

\end{document}